\documentclass[11pt, notitlepage]{article}
\usepackage{amssymb,amsmath,comment}
\catcode`\@=11 \@addtoreset{equation}{section}
\def\thesection{\arabic{section}}

\def\theequation{\thesection.\arabic{equation}}
\catcode`\@=12
\usepackage{colortbl}%
\usepackage{a4wide}

\newcommand{\ds} {\displaystyle}
\newcommand{\e}{\epsilon}
\newcommand{\pa} {\partial}

\newcommand{\ga} {\gamma}

\newcommand{\Om} {\Omega}
\newcommand{\ra} {\rightarrow}

\newcommand{\De} {\Delta}
\newcommand{\la} {\lambda}
\newcommand{\La} {\Lambda}
\newcommand{\noi} {\noindent}

\newcommand{\mb} {\mathbb}
\newcommand{\mc} {\mathcal}

\usepackage[all]{xy}
\catcode`\@=11
\def\theequation{\@arabic{\c@section}.\@arabic{\c@equation}}
\catcode`\@=12

\def\QED{\hfill {$\square$}\goodbreak \medskip}

\newtheorem{Theorem}{Theorem}[section]
\newtheorem{Lemma}[Theorem]{Lemma}
\newtheorem{Proposition}[Theorem]{Proposition}
\newtheorem{Corollary}[Theorem]{Corollary}

\newtheorem{Definition}[Theorem]{Definition}

\begin{document}
{\vspace{0.01in}}

\title
{ \sc Critical growth fractional elliptic problem with singular nonlinearities}

\author{
{\bf  Tuhina Mukherjee\footnote{email: tulimukh@gmail.com}}\; and\; {\bf K. Sreenadh\footnote{e-mail: sreenadh@gmail.com}}\\
{\small Department of Mathematics}, \\{\small Indian Institute of Technology Delhi}\\
{\small Hauz Khaz}, {\small New Delhi-16, India}\\
 }

\date{}

\maketitle

\begin{abstract}

\noi In this article, we study the following fractional Laplacian equation with critical growth and singular nonlinearity
\begin{equation*}
 \quad (-\De)^s u = \la a(x)u^{-q} + u^{{2^*_s}-1}, \quad u>0 \; \text{in}\;
\Om,\quad u = 0 \; \mbox{in}\; \mb R^n \setminus\Om,
\end{equation*}
where  $\Om$ is a bounded domain in $\mb{R}^n$ with smooth boundary $\partial \Om$, $n > 2s,\; s \in (0,1),\; \la >0,\; 0 < q \leq 1 $, $\theta \leq a(x) \in L^{\infty}(\Om)$, for some $\theta>0$ and $2^*_s=\frac{2n}{n-2s}$. We use variational methods to show the existence and multiplicity of positive solutions of the above problem with respect to the parameter $\la$.
\medskip

\noi \textbf{Key words:} Nonlocal operator, fractional Laplacian, Singular nonlinearities.

\medskip

\noi \textit{2010 Mathematics Subject Classification:} 35R11, 35R09, 35A15.

\end{abstract}

\section{Introduction}
Let $\Om \subset \mb R^n$ be a bounded domain with smooth boundary $\partial \Om$, $n>2s$ and $s \in (0,1)$. We consider the following problem with singular nonlinearity :
\begin{equation*}
(P_{\la}): \quad
 \quad (-\De)^s u = \la a(x)u^{-q} + u^{2^*_s-1}, \quad u>0 \; \text{in}\;
\Om, \quad u = 0 \; \mbox{in}\; \mb R^n \setminus\Om,
\quad
\end{equation*}
where $\la >0,\; 0 < q \leq 1 $, $\theta \leq a(x) \in L^{\infty}(\Om)$ for some $\theta > 0$, $2^*_s=\frac{2n}{n-2s}$ and $(-\De)^s$ is the fractional Laplace operator defined as
$$ (-\De)^s u(x) = - \frac{1}{2}\int_{\mb R^n} \frac{u(x+y)+u(x-y)-2u(x)}{|y|^{n+2s}}~dy , \; \; \text{for all }\; x \in \mb R^n. $$
The fractional power of Laplacian is the infinitesimal generator of L$\acute{e}$vy stable diffusion process and arise in anomalous diffusion in plasma, population dynamics, geophysical fluid dynamics, flames propagation, chemical reactions in liquids and American options in finance. For more details, we refer to \cite{da,gl}.

\noi Recently, the study of the fractional elliptic equations attracted  lot of interest by researchers in nonlinear analysis.
There are many works on existence of a solution for fractional elliptic equations with regular nolinearities like $u^q+\la u^p, \; p,\; q>0$. The sub critical growth problems are studied in  \cite{XcT,s1,s2}  and critical exponent problems are studied in \cite{bc,mb,mb1,s3,sv}. Also, the multiplicity of solutions by  the method of Nehari manifold and fibering maps has been investigated in \cite{ss1, ss2,xy,zlh}. To the best of our knowledge, there are no works dealing with multiplicity results with singular and critical nonlinearities. We also refer \cite{Aut,GR1,JPS,md1,md2,mb3,mb2,puc} for related works with fractional Laplacian, singular nonlinearities, critical growth or critical exponential nonlinearities. In this paper, we attempt to address the multiplicity of positive solutions of problem with singular type nonlinearity $\la u^{-q}+ u^{2^*_s-1},\; 0<q\leq1$.

\medskip
\noi In the local setting ($s=1$), the paper by Crandall, Rabinowitz and Tartar \cite{crt} is the starting point on semilinear problem with singular nonlinearity. A lot of work has been done related to existence and multiplicity results on singular nonlinearity. Among them we cite the reader to \cite{haitao,GR,NR,HFV,HKS,hcn,Hirano} and references therein. In \cite{hcn}, authors studied the critical growth singular problem
\[
-\De u = \la u^{-q}+ u^{2^*-1},\quad u>0 \; \text{in}\; \Omega, \quad u=0 \; \text{on}\; \pa \Om,
\]
where $0<q<1$. Using the variational methods and the geometry of the Nehari manifold, they proved the existence of multiple solutions in a suitable range of $\la$. Among the works dealing with elliptic equations with singular and critical growth terms, we cite also  \cite{AJ, DJP, haitao, Hirano} and references there in, with no attempt to provide a complete list.
\medskip

\noi
The fractional elliptic problem with only singular nonlinear term is studied by Fang \cite{fa} where author studied the following problem
 \begin{equation*}
 (-\De)^s u = u^{-p}, \quad u>0  \; \text{in}\;
\Om, \quad u = 0 \; \mbox{in}\; \mathbb{R}^n\backslash\Om,
\end{equation*}
with $0<p<1$. Here, authors used the method of sub and super solutions to show the existence of solution. Recently, in \cite{peral} the authors considered the problem
\begin{equation*}
(-\De)^s u = \la\frac{f(x)}{u^\ga} + M u^{p}, \quad
 u>0\;\text{in}\;\Om, \quad
 u = 0 \; \mbox{in}\; \mb R^n \setminus\Om,
\end{equation*}
where  $n>2s$, $M\ge 0$, $0<s<1$, $\ga>0$, $\la>0$, $1<p<2_{s}^{*}-1$ and $f\in L^{m}(\Om)$, $m\geq 1$ is a nonnegative function. Here, authors studied the existence of distributional solutions using the uniform estimates of  $\{u_n\}$ which are solutions of the regularized problems with singular  term $u^{-\ga}$ replaced by $(u+\frac{1}{n})^{-\ga}$.
They also discussed the multiplicity results when $M>0$ and for small $\lambda$ in  the sub critical case.

\medskip

\noi In this paper, we study the multiplicity results with convex-concave type critical growth and singular nonlinearity. Here, we follow the approach as in the work of Hirano, Saccon and Shioji \cite{hcn}. We obtain our results by studying the existence of minimizers that arise out of structure of Nehari manifold. We would like to remark that the results proved here are new even for the case $q=1$. Also, the multiplicity result is sharp in the sense that we consider the maximal range of $\lambda$ for which the corresponding fibering maps have two critical points.

\medskip
\noi The paper is organized as follows: In section 2, we present some preliminaries on function spaces required for variational settings. In section 3, we study the corresponding Nehari manifold and properties of minimizers. In section 4 and 5, we show the existence of minimizers and solutions. In section 6, we show some regularity results.
\section{Preliminaries and Main Results}
We recall some definitions of function spaces and results that are required in later sections.
\noi In \cite{sv}, Servadei and  Valdinoci discussed the Dirichlet
boundary value problem in case of fractional Laplacian using the variational techniques.  Due to nonlocalness of the fractional
Laplacian, they introduced the function space $(X_0,\|.\|_{X_0})$.
The space $X$ is defined as
\[X= \left\{u|\;u:\mb R^n \ra\mb R \;\text{is measurable},\;
u|_{\Om} \in L^2(\Om)\;
 \text{and}\;  \frac{(u(x)- u(y))}{ |x-y|^{\frac{n}{2s}+s}}\in
L^2(Q)\right\},\]
\noi where $Q=\mb R^{2n}\setminus(\mc C\Om\times \mc C\Om)$ and
 $\mc C\Om := \mb R^n\setminus\Om$. The space X is endowed with the norm defined as
\begin{align*}
 \|u\|_X = \|u\|_{L^2(\Om)} + \left[u\right]_X, \quad \text{where}\; \left[u\right]_X= \left( \int_{Q}\frac{|u(x)-u(y)|^{2}}{|x-y|^{n+2s}}dx
dy\right)^{\frac12}.
\end{align*}
 Then we define $ X_0 = \{u\in X : u = 0 \;\text{a.e. in}\; \mb R^n\setminus \Om\}$. Also, there exists a constant $C>0$ such that $\|u\|_{L^{2}(\Om)} \le C [u]_X$, for all $u\in X_0$. Hence,  $\|u\|=[u]_X$ is a norm on $(X_0, \|.\|)$ and  $X_0$ is a Hilbert space. Note that the norm $\|.\|$ involves the interaction between $\Om$ and $\mb R^n\backslash\Om$. We denote $\|.\|_{L^p(\Om)}$ as $\|.\|_p$ and $\|.\|=[.]_X$ for the norm in $X_0$.

%
\noi Now for each $\alpha \geq 0$, we set
\begin{equation}\label{eq00}
 C_{\alpha} = \sup \left \{ \int_{\Om} |u|^\alpha dx : \|u\| = 1\right \}.
 \end{equation}
Then $C_0  = |\Om| $ = Lebesgue measure of $\Om$ and
$\int_{\Om} |u|^\alpha dx \leq C_\alpha \|u\|^{\alpha}$, for all $ u \in X_0$.
 In the case of $n > 2s$, we set $2^*_s = \frac{2n}{ n-2s}$ , $0<s<1$. From the embedding results, we know that $X_0$ is continuously and compactly embedded in $L^r(\Om)$  when $1\leq r < 2^*_s$ and the embedding is continuous but not compact if $r= 2^*_s$. We define
 \begin{equation*}
 S = \inf_{u \in X_0\setminus \{0\}} \frac{\int_Q \frac{|u(x)-u(y)|^2}{|x-y|^{n+2s}}dxdy}{\left(\int_{\Om}|u|^{2^*_s}\right)^{2/2^*_s}} .
 \end{equation*}
\noi Consider the family of functions $\{U_{\epsilon}\}$, where $U_{\epsilon}$ is defined as
\[ U_{\epsilon} = \epsilon^{-(n-2s)/2}\; u^*\left(\frac{x}{\epsilon}\right),\; x \in \mb R^n\; \text{, for any} \; \epsilon>0, \]
where $u^*(x) = \bar{u}\left(\frac{x}{S^{\frac{1}{2s}}}\right),\; \bar{u}(x) = \frac{\tilde{u}(x)}{\|u\|_{2^*_s}}$ and $\tilde{u}(x)= \alpha(\beta^2 + |x|^2)^{-\frac{n-2s}{2}}$ with $\alpha \in \mb R \setminus \{0\}$ and $ \beta >0$ are fixed constants. Then for each $\epsilon > 0$, $U_\epsilon$ satisfies
\[ (-\De)^su = |u|^{2^*_s-2}u \; \;\text{in} \; \mb R^n \]
and verifies the equality
\[ \int_{\mb R^n} \frac{|U_\epsilon(x)-U_{\epsilon}(y)|^2}{|x-y|^{n+2s}} ~dxdy =\int_{\mb R^n} |U_\epsilon|^{2^*_s} = S^{\frac{n}{2s}}.\]
For a proof, we refer to \cite{s3}.

\begin{Definition}
We say $u$ is a positive weak solution of $(P_\la)$
if $u>0$ in $\Om$, $u \in X_0$ and
$$ \int_Q \frac{(u(x)-u(y))(\psi(x)-\psi(y))}{|x-y|^{n+2s}} ~dxdy - \int_\Om \left(\la a(x) u^{-q}  -  u^{2^{s}_{*}-1}\right) \psi ~dx = 0 \;\; \text{for all} \; \psi \in C^{\infty}_c(\Om).$$
\end{Definition}
We define the functional $I_{\la} : X_{0} \rightarrow (-\infty, \infty]$ by
\[ I_{\la}(u) = \frac12 \int_Q \frac{|u(x) - u(y)|^2}{|x-y|^{n+2s}}dxdy - \la \int_\Om a(x) G_q(u) dx - \frac{1}{2_{s}^{*}} \int_\Om |u|^{2_{s}^{*}}dx, \]
where $G_q : \mb{R} \rightarrow [-\infty, \infty)$ is the function defined by
$$G_q(x)=
\left\{
	\begin{array}{ll}
		\frac{|x|^{1-q}}{1-q}  & \mbox{if } 0<q<1 \\
		\ln|x| & \mbox{if } q=1
	\end{array}
\right.$$
for $x \in \mb{R}$. For each $0 < q \leq 1$, we set
$\ds X_+ = \{ u \in X_0 : u \geq 0\}$ and
$$X_{+,q} = \{ u \in X_+ : u \not\equiv 0, G_q(u) \in L^1(\Om)\}.$$
Notice that $ X_{+,q} = X_+ \setminus {0}$ if $0<q<1$ and $X_{+,1} \neq \emptyset$ if $\partial \Om$ is, for example, of class $C^2$. We will need the following important lemma.
\begin{Lemma}\label{lem2.1}
For each $w \in X_{+}$, there exists a sequence $\{w_k\}$ in $X_{0}$ such that, $w_k \rightarrow w$ strongly in $X_0$, where $0 \leq w_1 \leq w_2 \leq \ldots$ and $w_k$ has compact support in $\Om$, for each k .
\end{Lemma}
\begin{proof}
Let $w \in X_+$ and $\{\psi_k\}$ be sequence in $C^{\infty}_{c}(\Om)$ such that $\psi_k$ is nonnegative and converges strongly to $w$ in $X_0$. Define $z_k = \min \{\psi_k, w\}$,  then $z_k \rightarrow w$ strongly to $w$ in $X_0$. Now, we set $w_1 = z_{r_1}$ where $r_1>0$ is such that $\|z_{r_1}-w\| \leq 1$. Then $\max \{w_1, z_m\} \rightarrow w$ strongly as $m \rightarrow \infty$, thus we can find $r_2>0$ such that $\|max\{w_1,z_{r_2}\}-w\| \leq 1/2$. We set $w_2=max\{w_1,z_{r_2}\}$ and get $\max \{w_2, z_m\} \rightarrow w$ strongly as $m \rightarrow \infty$. Consequently, by induction we set, $w_{k+1}= \max\{w_k,z_{r_{k+1}}\}$ to obtain the desired sequence, since we can see that $w_k \in X_0$  has compact support, for each $k$ and $\|max\{w_k,z_{r_{k+1}}\}-w\| \leq 1/(k+1)$ which says that $\{w_k\}$ converges strongly to $w$ in $X_0$ as $k \rightarrow \infty$. \QED
\end{proof}
\noi For each $u\in X_{+,q}$ we define the fiber map $\phi_u:\mb R^+ \rightarrow \mb R$ by $\phi_u(t)=I_\la(tu)$. Then we prove the following:
\begin{Theorem} \label{thm3.2}
Assume $0<q \leq 1$. In case $q=1$, assume also $X_{+,1} \neq \emptyset$. Let $\Lambda$ be a constant defined by
 $\Lambda = \sup \left\{\la >0:\text{ for each}  \; u\in X_{+,q}\backslash\{0\}, ~\phi_u(t)~  \text{has two critical points in}  ~(0, \infty) \right\}.$
Then $\La >0$.
\end{Theorem}
Using the variational methods on the Nehari manifold (see section 3), we will prove the following multiplicity result.
\begin{Theorem}\label{thm2.4}
For all $\la \in (0, \Lambda)$, $(P_\la)$ has two solutions $u_\la$ and $v_\la$ in $X_{+,q}$.
\end{Theorem}

%

\section{Nehari Manifold and Fibering Map analysis}
We denote $I_{\la} = I$ for simplicity. In this section, we describe the structure of the Nehari manifold associated to the functional $I$.  One can easily verify that the energy functional $I$ is not bounded below on the space $X_0$. But we will show that $I$ is bounded below on this Nehari manifold and we will extract solutions by minimizing the functional on suitable subsets. The Nehari manifold is defined as
\[ \mc N_{\la} = \{ u \in X_{+,q} | \left\langle I^{\prime}(u),u\right\rangle = 0 \}. \]
\begin{Theorem}
$I$ is coercive and bounded below on $\mc N_{\la}$.
\end{Theorem}
\begin{proof} Case (I) ($0<q<1$): Since $u\in \mc N_\la$, using the embedding of $X_0$ in $L^{1-q}(\Om)$, we obtain
\begin{equation*}
\begin{split}
I(u) & = \left(\frac{1}{2}-\frac{1}{2^{*}_{s}}\right)\|u\|^2- \la \left(\frac{1}{1-q}-\frac{1}{2^*_s}\right)\int_{\Om} a(x)|u|^{1-q}dx\\
& \geq c_1 \|u\|^2 - c_2 \|u\|^{1-q}
\end{split}
\end{equation*}
for some nonnegative constants $c_1$ and $c_2$. Thus, $I$ is coercive and bounded below on $\mc N_{\la}$.\\
\noi Case (II) ($q=1$): In this case, using the inequality $\ln|u| \leq |u|$  and $X_0 \hookrightarrow L^1(\Om)$ we obtain
\begin{equation}
\begin{split}
I(u) &=  \left(\frac{1}{2}-\frac{1}{2^{*}_{s}}\right)\|u\|^2- \la \left( \int_{\Om} a(x) (\ln |u| - 1)~dx\right)\\
& \geq \left(\frac{1}{2}-\frac{1}{2^{*}_{s}}\right)\|u\|^2- \la \left( \int_{\Om} a(x)\ln |u| ~dx \right)\\
& \geq c_1^{\prime} \|u\|^2 - c_2^{\prime} \|u\|
\end{split}
\end{equation}
for some nonnegative constants $c_1^{\prime}$ and $c_2^{\prime}$. This again implies that $I$ is coercive and bounded below on $\mc N_{\la}$.\QED
\end{proof}

\noi From the definition of fiber map $\phi_u$, we have
$$\phi_u(t)=
\begin{cases} \ds
		\frac{t^2}{2} \|u\|^2 - \frac{t^{1-q}}{1-q} \int_{\Om} a(x)|u|^{1-q} dx - \frac{t^{2_{s}^{*}}}{2_{s}^{*}} \int_{\Om} |u|^{2_{s}^{*}} dx & \; \text{if}\; ~  0<q<1 \\
		\ds \frac{t^2}{2} \|u\|^2 - \frac{\la}{1-q} \int_{\Om} a(x)\ln(t|u|) dx - \frac{t^{2_{s}^{*}}}{2_{s}^{*}} \int_{\Om} |u|^{2_{s}^{*}} dx  & \; \text{if}\;~ q=1.
	\end{cases}$$
which gives
\[ \phi^{\prime}_u(t) = t \|u\|^2 -\la t^{-q} \int_{\Om} a(x)|u|^{1-q} dx - t^{2_{s}^{*}-1} \int_{\Om}|u|^{2_{s}^{*}} dx \]
\[\text{and }\; \phi^{\prime \prime}_u(t) = \|u\|^2 + q \la t^{-q-1} \int_{\Om} a(x)|u|^{1-q} dx - (2_{s}^{*}-1) t^{2_{s}^{*}-2} \int_{\Om}|u|^{2_{s}^{*}} dx .\]
It is easy to see that the points in $\mc N_{\la}$ are corresponding to critical points of $\phi_{u}$ at $t=1$. So, it is natural to divide $\mc N_{\la}$ in three sets corresponding to local minima, local maxima and points of inflexion. Therefore, we define
\begin{align*}
\mc N_{\la}^{+} = & \{ u \in \mc N_{\la}|~ \phi^{\prime}_u (1) = 0,~ \phi^{\prime \prime}_u(1) > 0\}=  \{ t_0u \in \mc N_{\la} |\; t_0 > 0,~ \phi^{\prime}_u (t_0) = 0,~ \phi^{\prime \prime}_u(t_0) > 0\},\\
\mc N_{\la}^{-} = & \{ u \in \mc N_{\la} |~ \phi^{\prime}_u (1) = 0,~ \phi^{\prime \prime}_u(1) <0\}=  \{ t_0u \in \mc N_{\la} |\; t_0 > 0, ~ \phi^{\prime}_u (t_0) = 0, ~\phi^{\prime \prime}_u(t_0) < 0\}
\end{align*}
and $ \mc N_{\lambda}^{0}= \{ u \in \mc N_{\la} | \phi^{\prime}_{u}(1)=0, \phi^{\prime \prime}_{u}(1)=0 \}. $
\begin{Lemma}\label{lem3.2}
There exist $\la_*>0$ such that for each $u\in X_{+,q}\backslash\{0\}$, there is unique $t_{\max}, t_1$ and $t_2$ with the property that $t_1<t_{max}<t_2, $$t_1 u\in \mc N_{\la}^{+}$ and $ t_2 u\in \mc N_{\la}^{-}$, for all $\la \in (0,\la_*)$.
\end{Lemma}
\begin{proof}
Define $A(u)= \int_{\Om}a(x)|u|^{1-q}~dx$ and $B(u)= \int_{\Om}|u|^{2_{s}^{*}}$. Let $u \in X_{+,q}$ then we have
\begin{align*}
\frac{d}{dt}I(tu)
=& t\|u\|^2 - t^{-q} A(u) - t^{2_{s}^{*}-1} B(u)\\
=&t^{-q} \left (m_u(t) - \la A(u) \right )
\end{align*}
and we define $m_u(t) := t^{1+q} \|u\|^2 - t^{2_{s}^{*}-1+q} B(u)$. Since $\ds \lim_{t \rightarrow \infty} m_u(t) = - \infty$,
we can easily see that $m_u(t)$ attains its maximum at $t_{max} = \left [ \frac{(1+q)\|u\|^2}{(2_{s}^{*}-1+q) B(u)} \right]^{\frac{1}{2_{s}^{*}-2}} $ and
\[ m_u(t_{max}) = \left( \frac{2_{s}^{*}-2}{2_{s}^{*}-1+q} \right) \left( \frac{1+q}{2_{s}^{*}-1+q} \right)^{\frac{1+q}{2_{s}^{*}-2}} \frac{(\|u\|^2)^\frac{2_{s}^{*}-1+q}{2_{s}^{*}-2}}{(B(u))^\frac{1+q}{2_{s}^{*}-2}}. \]
Now, $u \in \mc N_{\la}$  if and only if  $m_u(t) =  \la A(u) $ and we see that
\begin{equation*}
\begin{split}
m_u(t_{max}) - \la A(u)dx
\geq ~&  m_u(t_{max}) - \la \|a\|_{\infty} \|u\|^{1-q}_{{1-q}}\\
\geq ~& \left( \frac{2_{s}^{*}-2}{2_{s}^{*}-1+q} \right) \left( \frac{1+q}{2_{s}^{*}-1+q} \right)^{\frac{1+q}{2_{s}^{*}-2}} \frac{(\|u\|^2)^\frac{2_{s}^{*}-1+q}{2_{s}^{*}-2}}{B(u)^\frac{1+q}{2_{s}^{*}-2}} - \la \|a\|_{\infty} \|u\|^{1-q}_{{1-q}}>0\\
\end{split}
\end{equation*}
if and only if $\la <  \left( \frac{2_{s}^{*}-2}{2_{s}^{*}-1+q} \right) \left( \frac{1+q}{2_{s}^{*}-1+q} \right)^{\frac{1+q}{2_{s}^{*}-2}} ({C_{2_{s}^{*}}})^\frac{-1-q}{2_{s}^{*}-2} (\|a\|_{\infty} {C_{1-q}})^{-1} = \la_* $(say), where $C_{\alpha}$ is defined in \eqref{eq00}.

 \noi Case(I) $(0<q<1)$: We can also see that $ m_u(t) = \la \int_{\Om} a(x) |u|^{1-q}dx$ if and only if $\phi^{\prime}_u (t) = 0$. So for $\la \in (0,\lambda_*)$, there exists exactly two points $0<t_1<t_2$ with $m^{\prime}_u(t_1)>0$ and $m^{\prime}_u(t_2)<0$ that is, $t_1u \in \mc N^{+}_{\la}$ and $t_2u \in \mc N^{-}_{\la}$. Thus, $\phi_u$ has a local minimum at $t=t_1$ and a local maximum at $t=t_2$, that is $\phi_{u}$ is decreasing in $(0,t_1)$ and increasing in $(t_1,t_2)$.

\noi Case(II)$(q=1)$: Since $\ds \lim_{t \rightarrow 0} \phi_{u}(t) = \infty$ and $\ds \lim_{t \rightarrow \infty} \phi_{u}(t) = - \infty$ with similar reasoning as above we obtain $t_1, t_2$.  
That is, in both cases $\phi_{u}$ has exactly two critical points $t_1$ and $t_2$ such that $0< t_1 < t_2$, $\phi^{\prime \prime}_{u}(t_1) > 0$ and  $\phi^{\prime \prime}_{u}(t_2) < 0$ that is $t_1u \in \mc N_{\la}^{+}$, $t_2u \in \mc N_{\la}^{-}$.\QED
\end{proof}
\begin{Corollary}
$\mc N_{\la}^{0} = \{0\}$ for all $ \la \in (0, \La)$.
\end{Corollary}
\begin{proof}
Let $u \not\equiv 0 \in \mc N_{\la}^{0}$. Then $u \in \mc N_{\la}^{0}$ implies $u \in \mc N_{\la}$ that is, $1$ is a critical point of $\phi_{u}$. Using previous results, we say that $\phi_{u}$ has critical points corresponding to local minima or local maxima. So, $1$ is the critical point corresponding to local minima or local maxima of $\phi_{u}$. Thus, either $u \in \mc N_{\la}^{+}$ or $u \in \mc N_{\la}^{-}$ which is a contradiction.\QED
\end{proof}
{\bf Proof of Theorem \ref{thm3.2}}:
From lemma \ref{lem3.2}, we see that $\Lambda$ is positive. If $I_\la(tu)$ has two critical points for some $\lambda=\lambda^*$, then $t\mapsto I_\la(tu)$ also has two critical points for all $\la<\la^*$.\QED

\noi We can show that $\mc N_{\la}^{+}$ and $\mc N_{\la}^{-}$ are bounded in the following way:

\begin{Lemma}\label{le01}
The following holds:
\begin{enumerate}
\item[$(i)$] $\sup \{ \|u\|: u \in \mc N_{\la}^{+}\} < \infty $
\item[$(ii)$] $\inf \{ \|v\|: v \in \mc N_{\la}^{-} \} >0$  and $ \sup \{ \|v\| : v \in \mc N_{\la}^{-} , I(v) \leq M\} < \infty$ for each $M > 0$.
\end{enumerate}
Moreover, $\inf I(\mc N_{\la}^{+}) > - \infty$ and $\inf I(\mc N_{\la}^{-}) > - \infty$.
\end{Lemma}
\begin{proof}
\begin{enumerate}
\item[$(i)$] Let $u \in \mc N_{\la}^{+}$. Then we have
\begin{equation*}
\begin{split}
0 & < \phi^{\prime \prime}_u(1) =
(2-2_{s}^{*}) \|u\|^2 + \la (2_{s}^{*}-1+q) \int_{\Om} a(x) |u|^{1-q} dx   \\
 & \leq (2-2_{s}^{*}) \|u\|^2 + \la (2_{s}^{*}-1+q) C_{1-q} \|a\|_{\infty} \|u\|^{1-q} .\\
\end{split}
\end{equation*}
 Thus we obtain
\[ \|u\| \leq \left ( \frac{\la (2_{s}^{*}-1+q) C_{1-q}\|a\|_{\infty} }{2_{s}^{*}-2} \right )^{\frac{1}{1+q}}.\]


\item[$(ii)$] Let $v \in \mc N_{\la}^{-}$. We have
\begin{equation*}
\begin{split}
0 & > \phi^{\prime \prime}_v(1) =
   (1+q) \|v\|^2 - (2_{s}^{*}-1+q) \int_{\Om}|v|^{2_{s}^{*}} dx  \\
 & \geq (1+q) \|v\|^2 - (2_{s}^{*}-1+q) C_{2_{s}^{*}}\|v\|^{2_{s}^{*}}.
\end{split}
\end{equation*}
Thus, we obtain
\[ \|v\| \geq \left ( \frac{1+q}{(2_{s}^{*}-1+q)C_{2_{s}^{*}}}   \right )^{\frac{1}{2_{s}^{*}-2}}\]
which implies that $\inf \{ \|v\| : v \in \mc N_{\la}^{-} \} >0$. If $I(v) \leq M$, similarly we have for $0<q<1$
\[  \frac{(2_{s}^{*}-2)}{2 \times 2_{s}^{*}} \|v\|^2 - \la \left( \frac{2_{s}^{*}-1+q}{2_{s}^{*}(1-q)} \right)  C_{1-q} \|a\|_{\infty} \|v\|^{1-q} \leq M. \]
Now for $q=1$, using $\ln(|v|) \leq |v|$, we obtain
\begin{equation*}
 M  \geq \frac{(2_{s}^{*}-2)}{2 \times 2_{s}^{*}} \|v\|^2 - \la \|a\|_{\infty} C_1 \|v\| + \frac{\la}{2^*_s}\|a\|_{1} \geq \frac{(2_{s}^{*}-2)}{2 \times 2_{s}^{*}} \|v\|^2 - \la \|a\|_{\infty} C_1 \|v\|
\end{equation*}
 which implies  $ \sup \{ \|v\| : v \in \mc N_{\la}^{-} , Iv \leq M\} < \infty$, for each $M > 0$. For $u \in \mc N_{\la}^{+}$, when $0<q<1$ we have
\[ I(u) \geq -\frac{(1+q)}{2(1-q)}  \|u\|^2 - \frac{(2_{s}^{*}-1+q)}{2_{s}^{*}(1-q)} C_{2_{s}^{*}} \|u\|^{2_{s}^{*}}\]
and when $q=1$, we have
\[I(u) \geq \frac{\|u\|^2}{2}- \la \|a\|_{\infty}|\Om|^{\frac{2^*_s-1}{2^*_s}}C_{2^*_s}^{\frac{1}{2^*_s}}\|u\| - \frac{C_{2^*_s}}{2^*_s}\|u\|^{2^*_s}. \]
So, using ($i$) we conclude that $\inf I(\mc N_{\la}^{+}) > - \infty$ and similarly, using ($ii$) we can show that $\inf I(\mc N_{\la}^{-}) > - \infty$.\QED
\end{enumerate}
\end{proof}

\begin{Lemma}\label{le03} Suppose $u$ and $v$ be minimizers of $I$ over $\mc N_{\la}^{+}$ and $\mc N_{\la}^{-}$ respectively. Then for each $ w \in X_{+}$,
\begin{enumerate}
\item  there exists $\epsilon_0 > 0$ such that $I(u +\epsilon w) \geq I(u)$ for each $ \epsilon \in [0, \epsilon_0]$, and
\item $t_{\epsilon} \rightarrow 1$  as $\epsilon \rightarrow 0^+$, where $t_{\epsilon}$ is the unique positive real number satisfying $t_{\epsilon} (v + \epsilon w) \in \mc N_{\la}^{-}.$
\end{enumerate}
\end{Lemma}
\begin{proof}
\begin{enumerate}
\item{
Let $w \in X_{+}$ that is $w \in X_0$ and $w \geq 0$. We set
$$\rho(\epsilon) = \|u+\epsilon w\|^2 + \la q \int_{\Om} a(x) |u+\epsilon w|^{1-q}~dx - (2^{*}_{s}-1) \int_{\Om} |u+\epsilon w|^{2^{*}_{s}}$$
for each $\epsilon \geq 0$. Then using continuity of $\rho$ and $\rho(0) = \phi^{\prime \prime}_u(1) >0$, since $u \in \mc{N}^{+}_{\la}$, there exist $\epsilon_0>0 $ such that $\rho(\epsilon)> 0$ for $\epsilon \in [0, \epsilon_0]$. Since for each $\epsilon > 0$, there exists $t_{\epsilon}^{\prime}>0$ such that $t_{\epsilon}^{\prime}(u + \epsilon w) \in \mc{N}^{+}_{\la}$, so $t_{\epsilon}^{\prime} \rightarrow 1$ as $\epsilon \rightarrow 0$ and for each $\epsilon \in [0, \epsilon_0]$, we have
\[ I(u + \epsilon w) \geq I(t_{\epsilon}^{\prime}(u + \epsilon w))\geq \inf I(\mc{N}^{+}_{\la})= I(u).\]
}
\item{
We define $h :(0, \infty)\times \mb R^3 \rightarrow \mb R  $ by
\[ h(t,l_1,l_2,l_3) = l_1t - \la t^{-q}l_2 - t^{2^*_s-1}l_3 \]
for $(t,l_1,l_2,l_3)\in (0, \infty)\times \mb R^3.$ Then, $h$ is a $C^{\infty}$ function. Also, we have
\begin{equation*}
\frac{dh}{dt}(1, \|v\|^2,\int_{\Om}a(x)|v|^{1-q}~dx ,\int_{\Om}|v|^{2^*_s}) = \phi^{\prime \prime}_v(1)<0
\end{equation*}
and for each $\epsilon > 0, \; h(t_{\epsilon}, \|v+\epsilon w\|^2, \int_{\Om}a(x)|v + \epsilon w|^{1-q}~dx ,\int_{\Om} |v|^{2^*_s}) = \phi^{\prime}_{v+\epsilon w}(t_\epsilon)=0$. Moreover,
\[ h(1, \|v\|^2, \int_{\Om}a(x) |v|^{1-q}~dx, \int_{\Om} |v|^{2^*_s}) = \phi^{\prime}_v(1) = 0.\]
Therefore, by implicit function theorem, there exists an open neighborhood $ A \subset (0, \infty)$ and $B \subset \mb R^3$ containing $1$ and $(\|v\|^2,\int_{\Om}a(x) |v|^{1-q}~dx, \int_{\Om}|v|^{2^*_s} )$ respectively such that  for all $y \in B$, $h(t,y) = 0 $ has a unique solution $ t = g(y)\in A $, where $g : B \rightarrow A$ is a continuous function. So, $(\|v+\epsilon w\|^2,\; \int_{\Om}a(x)|v+ \epsilon w|^{1-q}~dx, \int_{\Om}|v+ \epsilon w|^{2^*_s}) \in B$ and
\[ g \left( \|v+\epsilon w)\|^2,\; \int_{\Om}a(x)|v+ \epsilon w|^{1-q}~dx, \int_{\Om}|v+ \epsilon w|^{2^*_s} \right) = t_{\epsilon}  \]
\noi since $h(t_{\epsilon},\|v+\epsilon w)\|^2,\; \int_{\Om}a(x)|v+ \epsilon w|^{1-q}~dx, \int_{\Om}|v+ \epsilon w|^{2^*_s}) = 0$. Thus, by continuity of $g$, we obtain $t_{\epsilon} \rightarrow 1$ as $\epsilon \rightarrow 0^+$.}\QED
\end{enumerate}
\end{proof}

\begin{Lemma} Suppose $u$ and $v$ are minimizers of $I$ on $\mc N_{\la}^{+}$ and $\mc N_{\la}^{-}$ respectively. Then
for each $w \in X_{+}$, we have $u^{-q}w, v^{-q} w \in L^{1}(\Om)$ and
\begin{align}
&\int_Q \frac{(u(x)-u(y))(w(x)-w(y))}{|x-y|^{n+2s}}~ dxdy - \la \int_\Om a(x) u^{-q}w~dx - \int_\Om  u^{2^*_s-1}w  \geq 0, \label{eq4.2}\\
&\int_Q \frac{(v(x)-v(y))(w(x)-w(y))}{|x-y|^{n+2s}} ~dxdy - \la \int_\Om a(x) v^{-q}w~dx - \int_\Om  v^{2^*_s-1}w  \geq 0.\label{eq4.3}\end{align}
\end{Lemma}
\begin{proof}
Let $w \in X_{+}$. For sufficiently small $\epsilon > 0$, by lemma \ref{le03},
\begin{equation}\label{eq7}
\begin{split}
0  \leq \frac{I(u+\epsilon w) - I(u)}{\epsilon}
 = & \frac{1}{2\e} (\|u+\epsilon w\|^2- \|u\|^2) - \frac{\la}{\e} \int_{\Om} a(x)(G_q(u + \epsilon w) - G_q(u))~dx \\
& - \frac{1}{\e 2^*_s} \int_{\Om} (|u+\epsilon w|^{2^*_s} - |u|^{2^*_s}) .\\
\end{split}
\end{equation}
We can easily verify that
\begin{enumerate}
\item[($i$)] $\ds \frac{(\|u+ \epsilon w\|^2- \|u\|^2)}{\epsilon} \rightarrow 2\int_{Q} \frac{(u(x)-u(y))(w(x)-w(y))}{|x-y|^{n+2s}}~dxdy\;\; \text{as}\; \epsilon \rightarrow 0^+$,
\item[($ii$)] $\ds \int_{\Om} \frac{(|u+\epsilon w|^{2^*_s}-|u|^{2^*_s})}{\epsilon}  \rightarrow 2^*_s \int_{\Om} |u|^{2^*_s-1} w \;\;\text{as}\; \epsilon \rightarrow 0^+$
\end{enumerate}
which implies that $ a(x) \frac{(G_q(u + \epsilon w) - G_q(u))}{\epsilon} \in L^{1}(\Om)$. Also, for each $x \in \Om,$
$$\frac{G_q(u(x)+\epsilon w(x)) - G_q(u(x))}{\epsilon}=
\left\{
	\begin{array}{ll}
		\frac{1}{\epsilon} \left( \frac{|u+\epsilon w|^{1-q}(x) - |u|^{1-q}(x)}{1-q} \right) & \mbox{if } 0<q<1 \\
		\frac{1}{\epsilon} \left( \ln(|u+\epsilon w|(x)) - ln (|u|(x)) \right) & \mbox{if } q=1
	\end{array}
\right.$$
which increases monotonically as $\epsilon \downarrow 0$ and
$$\lim\limits_{\epsilon \downarrow 0} \frac{G_q(u(x)+\epsilon w(x)) - G_q(u(x))}{\epsilon} =
\left\{
	\begin{array}{ll}
	0 & \mbox{if} \; w(x)=0 \\
	(u(x))^{-q} w(x) & \mbox{if} \; w(x)>0 , u(x) > 0\\
	\infty & \mbox{if}\; w(x) > 0 , u(x) =0.
    \end{array}
\right.$$
So using monotone convergence theorem, we obtain $u^{-q}w \in L^1(\Om)$. Letting $\epsilon \downarrow 0$ in both sides of \eqref{eq7}, we obtain \eqref{eq4.2}.
Next, we will show these properties for $v$. For each $\epsilon >0 $, there exists $t_{\epsilon}>0$ with $t_{\epsilon}(v+\epsilon w) \in \mc N^-_\la$. By lemma \ref{le03}(2), for sufficiently small $\epsilon > 0$, there holds
\[ I(t_{\epsilon}(v+\epsilon w)) \geq I(v) \geq I(t_{\epsilon}v)\]
which implies $I(t_{\epsilon}(v+\epsilon w)) - I(v) \geq 0$ and thus, we have
\[\la  \int_{\Om} a(x) ( G_q(|v+\epsilon w|^{1-q}) - G_q(|v|^{1-q})) dx \leq \frac{t_{\epsilon}^{q}}{2} (\|v+\epsilon w\|^2 - \|v\|^2) - \frac{t^{q+2^*_s}_\epsilon}{ 2^*_s} \int_{\Om}(|v+\epsilon w|^{2^*_s} - |v|^{2^*_s}).\]
As $\epsilon \downarrow 0$, $t_\epsilon \rightarrow 1$. Thus, using similar arguments as above, we obtain $v^{-q}w \in L^1(\Om)$ and \eqref{eq4.3} follows. \QED
\end{proof}

\noi Let $\phi_1>0$ be the eigenfunction of $(-\De)^s$ corresponding to the smallest eigenvalue $\la_1$. Then, $\phi_1 \in L^{\infty}(\Om)$ (see \cite{s3}) and
 \begin{equation*}
 \quad (-\De)^s \phi_1 = \la_1 \phi_1, \quad u>0\; \text{in}\;
\Om,\quad   \phi_1 = 0 \; \mbox{on}\; \mb R^n \setminus\Om.
\end{equation*}
For instance,  here we assume $\|\phi_1\|_{\infty} = 1$. Let $\eta > 0$ be such that $\phi = \eta \phi_1$ satisfies
 \begin{equation} \label{eq3.4} (-\De)^s \phi + \la a(x)\phi^{-q} + \phi^{2^*_s-1} > 0 \end{equation}
 and  $\phi^{2^*_s-1+q} (x) \leq \la a(x) \left(\frac{ q}{2^*_s-1}\right)$, for each $x \in \Om$. Then we have the following Lemma
 \begin{Lemma}\label{le05} Suppose $u$ and $v$ are minimizers of $I$ on $\mc N_{\la}^{+}$ and $\mc N_{\la}^{-}$ respectively. Then
  $u \geq \phi$ and $v\ge \phi.$
 \end{Lemma}
\begin{proof}
By lemma \ref{lem2.1},  let $\{w_k\}$ be a sequence in $X_0$ such that supp$(w_k)$ is compact, $0 \leq w_k \leq (\phi - u)^+$ for each $k$ and $\{w_k\}$ strongly converges to $(\phi - u)^+$ in $X_0$. Then
for each $x \in \Om$,
\begin{equation}\label{eq3.5}
\frac{d}{dt}(\la a(x)t^{-q}+t^{2^*_s-1})= -q \la a(x)t^{-q-1} +(2^*_s-1)t^{2^*_s-2} \leq 0\end{equation}
if and only if $t^{2^*_s-1+q}\leq \la a(x) \left(\frac{q}{2^*_s-1}\right).$
Using previous lemma and \eqref{eq3.4}, we have
\begin{align*}
 &\left(\int_Q \frac{(u(x)-u(y))(w_k(x)-w_k(y))}{|x-y|^{n+2s}} ~dxdy - \la \int_\Om a(x) u^{-q}w_k~dx - \int_\Om u^{2^*_s-1}w_k \right)\\
 &\quad - \left(\int_Q \frac{(\phi(x)-\phi(y))(w_k(x)-w_k(y))}{|x-y|^{n+2s}}~ dxdy - \la \int_\Om a(x) \phi^{-q}w_k~dx - \int_\Om \phi^{2^*_s-1}w_k\right) \geq 0\end{align*}
which implies
\begin{align*} \int_Q \frac{(u(x)-u(y))- (\phi(x)-\phi(y))}{|x-y|^{n+2s}} [w_k(x) - w_k(y)] ~dxdy
&-  \int_\Om (\la a(x)u^{-q} + u^{2^*_s-1})w_k~dx \\
&+ \int_{\Om} (\la a(x)\phi^{-q} +\phi^{2^*_s-1})w_k~dx \geq 0.\end{align*}
Using the strong convergence, we assume $\{w_k\}$ converges to $(\phi - u)^+$ pointwise almost everywhere in $\Om$ and we write $w_k(x) = (\phi - u)^+(x)+ o(1)$ as $k \rightarrow \infty$. Consider,
\begin{equation*}
\begin{split}
\int_Q &\frac{((u-\phi)(x)-(u - \phi)(y))}{|x-y|^{n+2s}} [w_k(x) - w_k(y)] ~dxdy\\
& = \int_Q \frac{((u-\phi)(x)-(u - \phi)(y))}{|x-y|^{n+2s}} ((\phi - u)^+(x) - (\phi - u)^+(y)) ~dxdy \\
&\quad \quad + o(1)\int_Q \frac{((u-\phi)(x)-(u - \phi)(y))}{|x-y|^{n+2s}}~dxdy
\end{split}
\end{equation*}
where we can see that
{\small
\begin{equation*}
\begin{split}
& \int_Q \frac{((u-\phi)(x)-(u - \phi)(y))}{|x-y|^{n+2s}} ((\phi - u)^+(x) - (\phi - u)^+(y)) ~dxdy\\
& = \int_Q \frac{-(\phi-u)(x)+(\phi-u)(y)}{|x-y|^{n+2s}}((\phi - u)^+(x) - (\phi - u)^+(y)) ~dxdy\\
& = \int_Q \frac{((\phi-u)^{+}-(\phi-u)^{-})(y)-((\phi-u)^{+}-(\phi-u)^{-})(x)}{|x-y|^{n+2s}}((\phi-u)^{+}(x)-(\phi-u)^{+}(y))~dxdy\\
& \leq -\int_Q \frac{|(\phi-u)^{+}(y)|^{2}+|(\phi-u)^{+}(x)|^2+(\phi-u)^{-}(y)(\phi-u)^{+}(x)+(\phi-u)^{-}(x)(\phi-u)^{+}(y)}{|x-y|^{n+2s}}~dxdy\\
& = - \int_Q \frac{|(\phi-u)^{+}(x) - (\phi-u)^{+}(y)|^2}{|x-y|^{n+2s}}~dxdy = - \| (\phi - u)^{+}\|^{2}.\\
\end{split}
\end{equation*}}
\noi Since $\phi^{2^*_s-1+q} (x) \leq \la a(x) \left(\frac{q}{2^*_s-1}\right)$ for each $x \in \Om$, using \eqref{eq3.5} we obtain
\begin{equation*}
\begin{split}
& \int_\Om ((\la a(x)u^{-q} + u^{2^*_s-1})- (\la a(x)\phi^{-q} + \phi^{2^*_s-1}))w_k~dx \\
& = \int_{\Om \cap \{\phi \geq u\}} ((\la a(x)u^{-q} + u^{2^*_s-1})- (\la a(x)\phi^{-q} + \phi^{2^*_s-1}))(\phi-u)^{+}(x)~dx + o(1) \geq 0.\\
 \end{split}
\end{equation*}
This implies
 \begin{align*}
 0 &\leq -\|(\phi-u )^+\|^2 -\int_\Om (\la a(x)u^{-q} + u^{2^*_s-1})w_k~dx + \int_{\Om} (\la a(x)\phi^{-q} + \phi^{2^*_s-1})w_k~dx +o(1)\\
  &\leq -\|(\phi-u )^+\|^2 + o(1)
  \end{align*}
and letting $k \rightarrow \infty$, we obtain
$-\|(\phi-u )^+\|^2 \geq 0.$
Thus, we showed $u \geq \phi$. Similarly, we can show $v\ge \phi.$ \QED
\end{proof}

\section{Existence of minimizer on $\mc N_{\la}^{+}$ }
In this section, we will show that the minimum of $I$ is achieved in $\mc N_{\la}^{+}.$ Also, we show that this minimizer is also the first solution of $(P_\la).$
\begin{Proposition}
For all $\la\in (0,\La)$, there exist $u_\la \in \mc N_{\la}^{+}$ satisfying $I(u_\la) = \inf\limits_{u\in \mc N_{\la}^{+}} I(u)$.
\end{Proposition}
\begin{proof}
Assume $0<q\leq1$ and $\la \in (0, \Lambda)$. Let $\{u_{k}\} \subset \mc N_{\la}^{+}$ be a sequence such that $I(u_{k}) \rightarrow \inf\limits_{u\in \mc N_{\la}^{+}} I(u)$ as $k \rightarrow \infty$. Using lemma \ref{le01}, we can assume that there exist $u_\la$ such that $u_{k} \rightharpoonup u_\la$ weakly  in $X_0$. First we will show that $\inf\limits_{u\in \mc N_{\la}^{+}} I(u) < 0$. Let $u_0 \in \mc N_{\la}^{+}$, then we have $\phi^{\prime \prime}_{u_0}(1) >0$ which gives
\[ \left( \frac{1+q}{2^{*}_{s}-1+q}  \right)\|u_0\|^2 > \int_{\Om}|u_0|^{2^{*}_{s}} dx .\]
Therefore, using $2^{*}_{s}-1>1$ we obtain
\begin{equation*}
\begin{split}
I(u_0) & = \left( \frac12 - \frac{1}{1-q} \right) \|u_0\|^2 + \left( \frac{1}{1-q} - \frac{1}{2^{*}_{s}} \right)\int_{\Om} |u_0|^{2^{*}_{s}}dx\\
& \leq -  \frac{(1+q)}{2(1-q)} \|u_0\|^2 + \frac{(1+q)}{2^{*}_{s}(1-q)} \|u_0\|^2= \left( \frac{1}{2^*_s}-\frac{1}{2}\right)\left(\frac{1+q}{1-q}\right) \|u_0\|^2<0.\\
\end{split}
\end{equation*}\\
We set $w_k := (u_k - u_\la)$ and claim that $u_k \rightarrow u_\la$ strongly in $X_0$. Suppose $\|w_k\|^2 \rightarrow c^2 \neq 0$ and $\int_{\Om} |w_k|^{2^*_s}dx \rightarrow d^{2^*_s}$ as $k \rightarrow \infty$. Since $u_k \in \mc N^+_{\la}$, using Brezis-Lieb we obtain
\begin{equation}\label{eq8}
0  = \lim_{k \rightarrow \infty} \phi^{\prime}_{u_k}(1) = \phi^{\prime}_{u_\la}(1)+c^2 -d^{2^*_s}
\end{equation}
which implies
$$ \|u_\la\|^2+c^2 = \la \int_{\Om}a(x)|u_\la|^{1-q}dx + \int_{\Om}|u_k|^{2^*_s}dx +d^{2^*_s}.$$
We claim that $u_\la \in X_{+,q}$. Suppose $u_\la\equiv 0$. If $0<q<1$ and $c=0$ then $0 > \inf I(\mc N^+_{\la}) = I(0)=0$, which is a contradiction and if $c\neq 0$ then
\begin{equation}\label{eq4.2new}
\inf\limits_{u\in \mc N_{\la}^{+}} I(u)= I(0)+\frac{c^2}{2} - \frac{d^{2^*_s}}{2^*_s} = \frac{c^2}{2} - \frac{d^{2^*_s}}{2^*_s} .\end{equation}
But we have $\|u_k\|^2_{2^*_s}S \leq \|u_k\|^2$ which gives $c^2 \geq Sd^2 $. Also from \eqref{eq8}, we have $c^2=d^{2^*_s}$. Then \eqref{eq4.2new} implies
\[ 0 > \inf\limits_{u\in \mc N_{\la}^{+}} I(u) = \left(\frac{1}{2}-\frac{1}{2^*_s}\right)c^2 \geq \frac{s}{n}S^{\frac{n}{2s}},\]
which is again a contradiction. In the case $q = 1$, the sequence $\left \{ \int_{\Om} \ln(|u_{k}|) \right \}$ is bounded, since the sequence $\{I(u_{k})\}$ and $\{\|u_{k}\|\}$ is bounded, using Fatou's lemma and  for each $k, \; \ln(|u_{k}|) \leq u_{k},$ we obtain
\[ -\infty < \overline{\lim_{k \rightarrow \infty}} \int_{\Om} \ln(|u_{k}|) dx \leq  \int_{\Om} \overline{\lim_{k \rightarrow \infty}} \ln(|u_{k}|)dx =  \int_{\Om} \ln(|u_\la|)dx. \]
which implies $u_\la \not\equiv 0$. Thus, in both cases we have shown that $u_\la \in X_{+,q}$. So, there exists $0 < t_1 < t_2$ such that $\phi^{\prime}_{u_{\la}}(t_1)= \phi^{\prime}_{u_{\la}}(t_2) = 0$ and $t_{1}u_{\la} \in \mc N^{+}_{\la}$. Then, three cases arise: \\
(i) $t_2 < 1$,\\
(ii) $t_2 \geq 1$ and $\frac{c^2}{2}- \frac{d^{2^*_s}}{2^*_s} < 0$, and \\
(iii) $t_2 \geq 1$ and $\frac{c^2}{2}- \frac{d^{2^*_s}}{2^*_s} \geq 0$.\\
\noi Case (i) Let $h(t) = \phi_{u_{\la}}(t)+ \frac{c^2t^2}{2} - \frac{d^{2^*_s}t^{2^*_s}}{2^*_s}$, $t >0$. By \eqref{eq8}, we obtain $ h^{\prime}(1) = \phi^{\prime}_{u_{\la}}(1)+c^2-d^{2^*_s} = 0$ and
\begin{equation*}
 h^{\prime}(t_2) = \phi^{\prime}_{u_{\la}}(1)+t_2c^2-{t_2}^{2^*_s-1}d^{2^*_s} = {t_2}(c^2 - {t_2}^{2^*_s-2}d^{2^*_s}) > {t_2}(c^2 - d^{2^*_s})
 > 0
\end{equation*}
which implies that $h$ increases on $[t_2,1]$. Then we obtain
\begin{equation*}
\begin{split}
\inf\limits_{u\in \mc N_{\la}^{+}} I(u) &= \lim I(u_k) \geq \phi_{u_{\la}}(1) + \frac{c^2}{2}- \frac{d^{2^*_s}}{2^*_s}  = h(1) > h(t_2)\\
& =\phi_{u_{\la}}(t_2) + \frac{c^2t_{2}^{2}}{2}- \frac{d^{2^*_s}t_{2}^{2^*_s}}{2^*_s} \geq \phi_{u_{\la}}(t_2) + \frac{t_{2}^{2}}{2} (c^2 - d^{2^*_s})\\
& > \phi_{u_{\la}}(t_2) > \phi_{u_{\la}}(t_1) \geq \inf\limits_{u\in \mc N_{\la}^{+}} I(u),
\end{split}
\end{equation*}
which is a contradiction.\\
\noi Case (ii) In this case, since $\la \in (0, \Lambda)$, we have $(c^2/2 - d^{2^*_s}/{2^*_s}) < 0$ and $Sd^2 \leq c^2$. Also we see that, for each $u_0 \in \mc N^{+}_{\la}$
\begin{equation*}
\begin{split}
0  < \phi^{\prime \prime}_{u_0}(1)& = \|u_0\|^2 + q \la \int_{\Om} a(x)|u_0|^{1-q} dx - (2^*_s-1) \int_{\Om}|u_0|^{1+p}dx \\
 & = (1+q) \|u_0\|^2 +  (-q-2^*_s+1) \int_{\Om} |u_0|^{2^*_s} dx  \\
\end{split}
\end{equation*}
\noi  which implies $ (1+q)\|u_0\|^2  > (q+2^*_s-1)\int_{\Om}|u_0|^{2^*_s}dx =  (q+2^*_s-1)\|u_0\|^{2^*_s}_{{2^*_s}}dx $ \\
\noi or, $ C_{2^*_s}  \leq \left(\frac{1+q}{q+2^*_s-1}\right) \|u_0\|^{2-2^*_s}$
\noi or, $ \|u_0\|^{2}  \leq  \left(\frac{1+q}{q+2^*_s-1} \right )^{\frac{2}{2^*_s-2}} S^{\frac{2^*_s}{2^*_s-2}}$. Thus, we have
\[ \sup\{ \|u\|^2: u \in \mc N^+_{\la}\} \leq \left(\frac{2}{2^*_s}\right)^{\frac{2}{2^*_s-2}} S^{\frac{2^*_s}{2^*_s-2}} < c^2 \leq \sup\{ \|u\|^2: u \in \mc N^+_{\la}\},\]
which gives a contradiction. Consequently, in case (iii)  we have
$$ \inf\limits_{u\in \mc N_{\la}^{+}} I(u) = I(u_\la)+\frac{c^2}{2}-\frac{d^{2^*_s}}{2^*_s} \geq I(u_\la) = \phi_{u_\la}(1) \geq \phi_{u_\la}(t_1) \geq \inf I(\mc N^+_{\la}) .$$
Clearly, this holds when $t_1 = 1$ and $(c^2/2 - d^{2^*_s}/{2^*_s}) = 0$ which yields $c=0$ and $u_\la \in \mc N^+_{\la}$. Thus, $u_k \rightarrow u_\la$ strongly in $X_0$ and $I(u_\la) = \inf\limits_{u\in \mc N_{\la}^{+}} I(u)$. \QED
\end{proof}
\begin{Proposition}\label{prp4.2}
$u_\la$  is a positive weak solution of $(P_\la)$.
\end{Proposition}
\begin{proof}
Let $\psi \in C^{\infty}_c(\Om)$. By lemma \ref{le05}, since $\phi > 0$, we can find $ \alpha >0$ such that $u_\la \geq \alpha$ on support of $\psi$. Then $u+\e \psi\ge0$ for small $\e$.  With similar reasoning as  in the proof of lemma \ref{le03}, $ I(u_\la+\epsilon \psi) \geq I(u_\la)$ for sufficiently small $\epsilon >0$. Then we have
\begin{equation*}
\begin{split}
0 & \leq \lim \limits_{\epsilon \rightarrow 0} \frac{I(u_\la+\epsilon\psi) - I(u_\la)}{\epsilon}\\
&= \int_Q \frac{(u_\la(x)-u_\la(y)) (\psi(x)-\psi(y))}{|x-y|^{n+2s}}~ dxdy - \la \int_\Om a(x) u_{\la}^{-q}\psi dx - \int_\Om u_{\la}^{2^*_s-1}\psi ~dx.
\end{split}
\end{equation*}
Since $\psi \in C^{\infty}_c(\Om)$ is arbitrary, we conclude that $u_\la$ is a positive weak solution of $(P_\la)$.\QED
\end{proof}
As a consequence, we have
\begin{Lemma}\label{lem4.3}
$\La <\infty.$
\end{Lemma}
\begin{proof}
Taking $\phi_1$ as the test function in $(P_\la)$, we obtain
\[\la_1 \int_\Om u \phi_1 = \int_{\Om} (\la u^{-q}+u^{2_{s}^{*}-1})\phi_1.\]
Let $\mu^*>0$ be  such that $\mu^* t^{-q} + t^{2_{s}^{*}-1} > (\la_1 +\e) t$, for all $t>0$. Then we obtain $\la <\mu^*$ and the proof follows. \QED
\end{proof}
\section{Existence of minimizer on $\mc N^{-}_{\la}$}

\noi In this section, we shall show the existence of second solution by proving the existence of minimizer of $I$ on $\mc N^-_{\la}$. We need some lemmas to prove this and for instance, we assume $0 \in \Om$. In order to put $U_\epsilon$ zero outside $\Om$, we fix $\delta >0$ such that $B_{4\delta} \subset \Om$ and let $\zeta \in C^{\infty}_c(\mb R^n)$ be such that $0 \leq \zeta \leq 1$ in $\mb R^n$, $\zeta \equiv 0$ in $\mc C B_{2\delta}$ and $\zeta \equiv 1$ in $B_\delta$. For each $\epsilon > 0$ and $x \in \mb R^n$, we define
$$ \Phi_\epsilon(x) :=  \zeta(x) U_\epsilon(x).$$ Moreover, since $u_\la$ is positive and bounded (see lemma \ref{lem6.2}), we can find $m,M>0$ such that for each $x \in \Om$, $m \leq u_\la(x)\leq M$.
\begin{Lemma}
$\sup \{I(u_\la + t\Phi_\epsilon): t\geq 0\} < I(u_\la) + \frac{s}{n}S^{\frac{n}{2s}}$, for each sufficiently small $\epsilon >0$.
\end{Lemma}
\begin{proof}
We assume $\epsilon>0$ to be sufficiently small. Using proposition $21$ of \cite{s3}, we have
$$ \int_{Q} \frac{|\Phi_\epsilon(x) - \Phi_\epsilon(y)|^2}{|x-y|^{n+2s}}dxdy \leq S^\frac{n}{2s}+ o(\epsilon^{n-2s})$$
 which says that we can find $r_1 > 0$ such that
 $$\int_{Q} \frac{|\Phi_\epsilon(x) - \Phi_\epsilon(y)|^2}{|x-y|^{n+2s}}dxdy \leq S^\frac{n}{2s}+ r_1\epsilon^{n-2s} .$$
Now, we consider

\begin{equation*}
\begin{split}
\int_{\Om} |\Phi_\epsilon|^{2^*_s}dx & =\int_{B_\delta} |U_\epsilon|^{2^*_s}dx +\int_{B_{2\delta} \setminus B_{\delta}} |\zeta(x)U_\epsilon(x)|^{2^*_s}dx \\
& = \int_{\mb R^n} |U_\epsilon|^{2^*_s}dx - \int_{\mb R^n \setminus B_\delta} |U_\epsilon|^{2^*_s}dx + \int_{B_{2\delta} \setminus B_{\delta}} |\zeta(x)U_\epsilon(x)|^{2^*_s}dx\\
& \geq \int_{\mb R^n} |U_\epsilon|^{2^*_s}dx - \int_{\mb R^n \setminus B_\delta} |U_\epsilon|^{2^*_s}dx\\
& = S^\frac{n}{2s}- \epsilon^{-n}\int_{\mb R^n \setminus B_\delta} |u^*(x/\epsilon)|^{2^*_s}dx \geq  S^\frac{n}{2s} - r_2\epsilon^n
\end{split}
\end{equation*}
 for some constant $r_2 >0$. We now fix $1 < \rho < \frac{n}{n-2s}$ and we have
 \begin{equation*}
\begin{split}
\int_{\Om}|\Phi_\epsilon|^\rho dx &= \epsilon^{-(n-2s)\rho/2}\int_{B_{2\delta}}|\zeta(x)u^*(x/\epsilon)|^\rho dx\\
& \leq {r_3}^{\prime} \epsilon^{-(n-2s)\rho/2} \int_{B_{2\delta}} \left( \left|\frac{x}{\epsilon S^\frac{1}{2s}}\right|^2\right)^{\frac{-(n-2s)\rho}{2}}dx\\
& = r_3 \epsilon^{(n-2s)\rho/2}
\end{split}
\end{equation*}
for constants ${r_3}^{\prime}, r_3 >0$. Now, choosing $\epsilon >\delta/S^\frac{1}{2s}$ we see that
 \begin{equation*}
\begin{split}
 \int_{B_\delta} |\Phi_\epsilon|^{2^*_s-1}dx &= \alpha^{2^*_s-1}\beta^{-(n+2s)} \epsilon^{(n-2s)/2} \int_{|y|< \frac{\delta}{\epsilon} S^\frac{1}{2s}}(1+|y|^2)^{-(n+2s)/2}dy\\
 & \geq \alpha^{2^*_s-1}\beta^{-(n+2s)} \frac{\epsilon^{(n-2s)/2}}{2^{(n+2s)/2}}\int_{|y|\leq 1}dy = r_4\epsilon^{(n-2s)/2}
 \end{split}
\end{equation*}
for some constant $r_4>0$. We can find appropriate constants $\rho_1, \rho_2 >0$ such that the following inequalities holds :
\begin{enumerate}
\item  $\displaystyle \la \left( \frac{(c+d)^{1-q}}{1-q} - \frac{c^{1-q}}{1-q} - \frac{d}{c^q} \right) \geq -\frac{\rho_1d^{\rho}}{r_3},$ for all $c \geq m, d\geq 0$,
\item
$\displaystyle \left(\frac{(c+d)^{2^*_s}}{2^*_s} - \frac{c^{2^*_s}}{2^*_s}-c^{2^*_s-1}d\right)\geq \frac{d^{2^*_s}}{2^*_s},$  for all $c,d \geq 0$,
\item
$ \displaystyle \frac{(c+d)^{2^*_s}}{2^*_s}-\frac{c^{2^*_s}}{2^*_s} - {c^{2^*_s-1}}d \geq \frac{d^{2^*_s}}{2^*_s} + \frac{\rho_2cd^{2^*_s-1}}{r_4m(2^*_s-1)},$ for all $0\leq c \leq M,\; d\geq 1$.
\end{enumerate}
 Case ($I$)($0<q<1$): Since $u$ is a positive weak solution of $(P_\la)$, using above inequalities, we obtain
 {\small \begin{equation*}
 \begin{split}
 I(u_\la +t\Phi_\epsilon)-I(u_\la)
 & = I(u_\la+t\Phi_\epsilon)-I(u_\la) - t\left(\int_Q \frac{(u_\la(x)-u_\la(y))(\Phi_\epsilon(x)-\Phi_\epsilon(y))}{|x-y|^{n+2s}}\right. ~dxdy \\
 &\hspace*{3.5cm}\left.-\la \int_\Om a(x)\Phi_\epsilon(x)u_{\la}^{-q}~dx - \int_{\Om}\Phi_\epsilon(x)u_{\la}^{2^*_s-1}~dx \right)\\
 & = \frac{t^2}{2} \int_{Q} \frac{|\Phi_\epsilon(x)- \Phi_\epsilon(y)|^2}{|x-y|^{n+2s}}~dxdy  -  \int_{\Om}( |u_\la+t\Phi_\epsilon|^{2^*_s}-|u_\la|^{2^*_s})~dx- t\int_{\Om}u_{\la}^{2^*_s-1}\Phi_\epsilon~dx  \\
 & \hspace*{2cm} - \la \left( \int_{\Om}\frac{a(x)(|u_\la+t\Phi_\epsilon|^{1-q}-|u_\la|^{1-q})}{1-q}~dx - t \int_{\Om}a(x)\Phi_\epsilon u_{\la}^{-q}~dx \right)\\
 & \leq \frac{t^2}{2}(S^{\frac{n}{2s}}+r_1\epsilon^{n-2s})-\frac{t^{2^*_s}}{2^*_s}\int_{\Om}|\Phi_\epsilon|^{2^*_s}dx + \frac{\rho_1t^{\rho}}{r_3}\int_{\Om}a(x)|\Phi_\epsilon|^{\rho}dx\\
 & \leq \frac{t^2}{2}(S^\frac{n}{2s}+r_1\epsilon^{n-2s})-\frac{\rho_2t^{2^*_s}}{2^*_s}(S^\frac{n}{2s}-r_2\epsilon^n)
 +\|a\|_{\infty}\rho_1t^{\rho}\epsilon^{(n-2s)\rho/2}
 \end{split}
\end{equation*}}
 for $0\leq t\leq 1/2$. Since we can assume $t\Phi_\epsilon \geq 1$, for each $t \geq 1/2$ and $|x| \leq 2\delta$, we have
 \begin{equation*}
 \begin{split}
 I(u_\la+t\Phi_\epsilon)-I(u_\la)
 & \leq \frac{t^2}{2}(S^\frac{n}{2s}+r_1\epsilon^{n-2s})-\frac{t^{2^*_s}}{2^*_s}\int_{\Om}|\Phi_\epsilon|^{2^*_s}dx -\frac{\rho_2t^{2^*_s-1}}{r_4(2^*_s-1)}\int_{|x|\leq \delta}|\Phi_\epsilon|^{2^*_s-1}dx \\
 &\quad + \frac{\rho_1t^{\rho}}{r_3}\int_{\Om}a(x)|\Phi_\epsilon|^{\rho}dx\\
 & \leq \frac{t^2}{2}(S^\frac{n}{2s}+r_1\epsilon^{n-2s})-\frac{t^{2^*_s}}{2^*_s}(S^\frac{n}{2s}-r_2\epsilon^n)-\frac{\rho_2 t^{2^*_s-1}}{(2^*_s-1)} \epsilon^{\frac{(n-2s)}{2}}\\
 &\quad +\frac{\rho_1t^{\rho}}{r_3}+\|a\|_{\infty}\rho_1t^{\rho}\epsilon^{\frac{(n-2s)\rho}{2}}.
\end{split}
\end{equation*}
Now, we define a function $h_\epsilon : [0, \infty) \rightarrow \mb R$ by
$$h_\e(t)=\begin{cases}		 \frac{t^2}{2}(S^\frac{n}{2s}+r_1\epsilon^{n-2s})-\frac{t^{2^*_s}}{2^*_s}(S^\frac{n}{2s}-r_2\epsilon^n)+\|a\|_{\infty}\rho_1t^{\rho}\epsilon^{\frac{(n-2s)\rho}{2}} &  t\in[0,1/2) \\ \left.\begin{array}{ll} \frac{t^2}{2}(S^\frac{n}{2s}+r_1\epsilon^{n-2s})-\frac{t^{2^*_s}}{2^*_s}(S^\frac{n}{2s}-r_2\epsilon^n)&-\frac{\rho_2t^{2^*_s-1}}{(2^*_s-1)} \epsilon^{\frac{(n-2s)}{2}}+\frac{\rho_1t^{\rho}}{r_3}  \\
 & +\|a\|_{\infty}\rho_1t^{\rho}\epsilon^{\frac{(n-2s)\rho}{2}}\end{array}\right\} & t\in [1/2,\infty)
	\end{cases}$$
With some computations, it can be checked that $h_\epsilon$ attains its maximum at
\[ t = 1-\frac{\rho_2\epsilon^{(n-2s)/2}}{(2^*_s-2)S^\frac{n}{2s}}+o(\epsilon^{(n-2s)/2}), \]
so we obtain
\[ \sup\{I(u_\la+t\Phi_\epsilon)-I(u_\la): t\geq 0 \} \leq \frac{s}{n}S^{\frac{n}{2s}}-\frac{\rho_2\epsilon^{(n-2s)/2}}{(2^*_s-1)}+o(\epsilon^{(n-2s)/2})< \frac{s}{n}S^{\frac{n}{2s}}. \]
\noi Case ($II$) ($q=1$): Since $u_\la$ is a positive weak solution of $(P_\la)$, using previous inequalities, we obtain
 \begin{equation*}
 \begin{split}
 I(u_\la+t\Phi_\epsilon)&-I(u_\la)
  = I(u_\la+t\Phi_\epsilon)-I(u_\la) - t\left(\int_Q \frac{(u_\la(x)-u_\la(y))(\Phi_\epsilon(x)-\Phi_\epsilon(y))}{|x-y|^{n+2s}}~dxdy \right.\\
 &\left.\hspace*{4.6cm}-\la \int_\Om a(x)\frac{\Phi_\epsilon(x)}{u_\la}~dx - \int_{\Om}\Phi_\epsilon(x)u_{\la}^{2^*_s-1}~dx\right)\\
 & = \frac{t^2}{2} \int_{Q} \frac{|\Phi_\epsilon(x)- \Phi_\epsilon(y)|^2}{|x-y|^{n+2s}}~dxdy  -  \int_{\Om}( |u_\la+t\Phi_\epsilon|^{2^*_s}-|u_\la|^{2^*_s})~dx- t\int_{\Om}u_{\la}^{2^*_s-1}\Phi_\epsilon~dx \\
 & \hspace*{3.5cm} - \la \left( \int_{\Om}a(x)( \ln |u_{\la}+t\Phi_\epsilon|- \ln |u_{\la}|)~dx - t\int_{\Om}a(x)\frac {\Phi_\epsilon}{u_\la}~dx \right)\\
 & \leq \frac{t^2}{2}(S^{\frac{n}{2s}}+r_1\epsilon^{n-2s})-\frac{t^{2^*_s}}{2^*_s}\int_{\Om}|\Phi_\epsilon|^{2^*_s}~dx - \la \int_{\Om} a(x) \frac{t^2|u_{\epsilon}|^2}{2|u_\la+tu_{\epsilon}|^2}~dx\\
 & \leq \frac{t^2}{2}(S^{\frac{n}{2s}}+r_1\epsilon^{n-2s})-\frac{t^{2^*_s}}{2^*_s}\int_{\Om}|\Phi_\epsilon|^{2^*_s}~dx + \frac{\la \|a\|_{\infty}}{2M^2}\int_{\Om}t^2|u_{\epsilon}|^2~dx\\
 & \leq \frac{t^2}{2}(S^{\frac{n}{2s}}+r_1\epsilon^{n-2s})-\frac{t^{2^*_s}}{2^*_s}\int_{\Om}|\Phi_\epsilon|^{2^*_s}~dx + \gamma \|a\|_{\infty} \epsilon^{(n-2s)}
 \end{split}
\end{equation*}
 for $0\leq t\leq 1/2$, where $\gamma$ is a constant. Since we can assume $t\Phi_\epsilon \geq 1$, for each $t \geq 1/2$ and $|x| \leq 2\delta$, we have
 \begin{equation*}
 \begin{split}
 I(u_\la+t\Phi_\epsilon)-I(u_\la)
 & \leq \frac{t^2}{2}(S^\frac{n}{2s}+r_1\epsilon^{n-2s})-\frac{t^{2^*_s}}{2^*_s}\int_{\Om}|\Phi_\epsilon|^{2^*_s}dx \\
 &\quad -\frac{\rho_2t^{2^*_s-1}}{r_4(2^*_s-1)}\int_{|x|\leq \delta}|\Phi_\epsilon|^{2^*_s-1}dx - \la \int_{\Om} a(x) \frac{t^2|u_{\epsilon}|^2}{2|u_\la+tu_{\epsilon}|^2}dx\\
 & \leq \frac{t^2}{2}(S^\frac{n}{2s}+r_1\epsilon^{n-2s})-\frac{t^{2^*_s}}{2^*_s}(S^\frac{n}{2s}-r_2\epsilon^n)-\frac{\rho_2t^{2^*_s-1}}{(2^*_s-1)} \epsilon^{\frac{(n-2s)}{2}}\\
 &\quad +\frac{\rho_1t^{\rho}}{r_3}  + \gamma \|a\|_{\infty} \epsilon^{(n-2s)}
\end{split}
\end{equation*}
for each $t \geq 1/2$. With similar computations as in  case (I), it can be shown that
\[ \sup \{I(u_\la + t\Phi_\epsilon): t\geq 0\} < I(u_\la) + \frac{s}{n}S^{\frac{n}{2s}}.\]\QED
\end{proof}
\begin{Lemma}
There holds $\inf I(\mc N^-_{\la}) < I(u_\la)+\frac{s}{n} S^{\frac{n}{2s}}$.
\end{Lemma}
\begin{proof}
We start by fixing sufficiently small $\epsilon >0$ as in previous lemma and define functions $\sigma_1, \sigma_2: [0,\infty) \rightarrow \mb R$ by
\[ \sigma_1(t) = \int_{Q}\frac{|(u_\la+t\Phi_\epsilon)(x)-(u_\la+t\Phi_\epsilon)(y)|^2}{|x-y|^{n+2s}}~dxdy - \la \int_{\Om}a(x)|u_\la+t\Phi_\epsilon|^{1-q}~dx - \int_{\Om}|u_\la+t\Phi_\epsilon|^{2^*_s}~dx,\]
\begin{align*} \sigma_2(t) = \int_{Q}\frac{|(u_\la+t\Phi_\epsilon)(x)-(u_\la+t\Phi_\epsilon)(y)|^2}{|x-y|^{n+2s}}~dxdy &- \la q \int_{\Om}a(x)|u_\la+t\Phi_\epsilon|^{1-q}~dx\\
& - (2^*_s-1) \int_{\Om}|u+t\Phi_\epsilon|^{2^*_s}~dx. \end{align*}
Let $t_0 = \sup\{t \geq 0: \sigma(t)\geq 0 \}$, then $\sigma_2(0) = \phi^{\prime\prime}_{u_\la}(1) >0$ and $\sigma_2(t) \rightarrow -\infty$ as $t \rightarrow \infty$ which implies $0<t_0<\infty$. As $\la \in (0, \Lambda)$, we obtain $\sigma_1(t_0)>0$ and since $\sigma_1(t) \rightarrow -\infty$ as $t \rightarrow \infty$, there exists $t^{\prime} \in (t_0, \infty)$ such that $\sigma_1(t^{\prime})=0.$ This gives $\phi^{\prime \prime}_{u_\la+t^{\prime}\Phi_\epsilon}(1)<0$, because $t^{\prime}>t_0$ implies $\sigma_2(t^{\prime})<0.$ Hence, $ (u_\la+t^{\prime}\Phi_\epsilon)\in \mc N^-_{\la}$ and using previous lemma, we obtain the result.\QED
\end{proof}

\begin{Proposition}
There exists $v_\la \in \mc N_{\la}^{-}$ satisfying $I(v_\la) = \inf\limits_{v\in \mc N_{\la}^{-}} I(v)$.
\end{Proposition}
\begin{proof}
Let $\{v_k\}$ be sequence in $\mc N^-_{\la}$ such that $I(v_k) \rightarrow \inf I(\mc N^-_{\la})$ as $k \rightarrow \infty$. Using lemma \ref{le01}, we may assume that there exist $v_\la$ such that $v_k \rightharpoonup v_\la$ weakly in $X_0$. We set $z_k := (v_k - v_\la)$ and claim that $v_k \rightarrow v_\la$ strongly in $X_0$. Suppose $\|z_k\|^2 \rightarrow c^2$ and $\int_{\Om}|z_k|^{2^*_s}dx \rightarrow d^{2^*_s}$ as $k \rightarrow \infty$. Then using Brezis-Lieb lemma, we obtain
\[ \|v_\la\|^2 +c^2 = \la\int_{\Om}a(x)|v_\la|^{1-q}dx + \int_{\Om}|v_\la|^{2^*_s}+d^{2^*_s}dx. \]
We claim that $v_\la \in X_{+,q}$. Suppose $v_\la =0$, this implies $c \neq 0$ (using lemma 3.4(ii)) and thus
\[ \inf I(\mc N^-_{\la}) = \lim I(v_k) = I(0)+\frac{c^2}{2}-\frac{d^{2^*_s}}{2^*_s} \geq \frac{s}{n}S^\frac{n}{2s},\]
but by previous lemma, $ (sS^\frac{n}{2s})/n \leq \inf I(\mc N^-_{\la}) < I(u_\la)+ \frac{s}{n}S^\frac{n}{2s}$ implying $I(u_\la) >0$ or we say that $\inf I(\mc N_{\la}^{-}) >0$, which is a contradiction. So $v_\la \in X_{+,q}$ and thus, our assumption $\la \in (0, \Lambda)$ says that there exists $0<t_1<t_2$ such that $\phi^{\prime}_{v_\la}(t_1)= \phi^{\prime}_{v_\la}(t_2)=0$ and $t_1 v_\la \in \mc N^+_{\la}$, $t_2 v_\la \in \mc N^-_{\la}$. Let us define $f,g :(0,\infty) \rightarrow \mb R$ by
\begin{equation}\label{eq11}
 f(t) = \frac{c^2t^2}{2}-\frac{d^{2^*_s}t^{2^*_s}}{2^*_s} \; \text{and} \; g(t) = \phi_{v_\la}(t)+ f(t).
\end{equation}
Then, following three cases arise :\\
\noi (i) $t_2 < 1,$\\
\noi(ii) $t_2 \geq 1$ and $d >0$, and \\
\noi(iii) $t_2 \geq 1$ and $d = 0$.

\noi Case (i) $t_2 <1$ implies $g^{\prime}(1)= \phi^{\prime}_{v_\la}(1)+ f^{\prime}(1) = 0$, using $\eqref{eq11}$ and $g^{\prime}(t_2)= \phi^{\prime}_{v_\la}(t_2)+ f^{\prime}(t_2) = t_2(c^2-d^{2^*_s}t_{2}^{2^*_s-2}) \geq t_2(c^2-d^{2^*_s})>0$. This implies that $g$ is increasing on $[t_2,1]$ and we have
\[ \inf I(\mc N^-_{\la}) = g(1)> g(t_2) \geq I(t_2 v_\la)+\frac{{t_2}}{2}(c^2-d^{2^*_s}) > I(t_2 v_\la) \geq \inf I(\mc N^-_{\la}) \]
which is a contradiction.\\
\noi Case (ii) Let $\underline{t}=(c^2/d^{2^*_s})^{\frac{1}{2^*_s-2}}$ and we can check that $f$ attains its maximum at $\underline{t}$ and $$f(\underline{t})= \frac{c^2{\underline{t}}^2}{2}-\frac{d^{2^*_s}{\underline{t}}^{2^*_s}}{2^*_s}=\left( {\frac{c^2}{d^{2^*_s}}}\right)^{\frac{2^*_s}{2^*_s-2}}\left(\frac{1}{2}- \frac{1}{2^*_s} \right) \geq S^{\frac{2^*_s}{2^*_s-2}}\frac{s}{n}= \frac{s}{n}S^\frac{n}{2s}. $$
Also, $f^{\prime}({t})= (c^2-d^{2^*_s}t^{2^*_s-2})t >0 $ if $0<t<\underline{t}$ and $f^{\prime}(t) <0$ if $t > \underline{t}$. Moreover, we know $g(1) = \max \limits_{t>0} \{g(t)\} \geq g(\underline{t})$ using the assumption $\la \in (0, \Lambda)$. If $\underline{t}\leq 1$, then we have
\[ \inf I(\mc N^-_{\la}) = g(1)\geq g(\underline{t})=I(\underline{t}v_\la)+ f(\underline{t}) \geq I(t_1 v_\la)+\frac{s}{n}S^\frac{n}{2s} \]
which contradicts the previous lemma. Thus, we must have $\underline{t}>1$. Since $g^{\prime}(t) \leq 0$ for $t\geq 1$, there holds $\phi^{\prime\prime}_{v_\la}(t) \leq -f^{\prime}(t) \leq 0 $ for $1\leq t \leq \underline{t}$. Then we have $\underline{t}\leq t_1$ or $t_2=1$. If $\underline{t}\leq t_1$ then
\[ \inf I(\mc N^-_{\la})= g(1) \geq  g(\underline{t})=I(\underline{t}v_\la)+ f(\underline{t}) \geq I(t_1v_\la)+\frac{s}{n}S^\frac{n}{2s}\]
which is a contradiction. If $t_2=1$ then using $c^2=d^{2^*_s}$ we obtain
\[ \inf I(\mc N^-_{\la})=g(1)= I(v_\la)+ \left(\frac{c^2}{2}-\frac{d^{2^*_s}}{2^*_s}\right)\geq I(v_\la) + \frac{s}{n}S^\frac{n}{2s} \]
which is a contradiction and thus only case (iii) holds. If $c \neq 0$, then $\phi^{\prime}_{v_\la}(1) = -c^2 <0$ and $\phi^{\prime \prime}_{v_\la}(1) = -c^2<0$ which contradicts $t_2 \geq 1$. Thus, $c=0$ which implies $v_k \rightarrow v_\la$ strongly in $X_0$. Consequently, $v_\la \in \mc N^-_{\la}$ and $\inf I(\mc N^-_{\la})= I(v_\la)$.\QED
\end{proof}

\begin{Proposition}\label{prp5.4}
 For $\la \in (0,\La),$ $v_\la$ is a positive weak solution of $(P_\la)$.
\end{Proposition}
\begin{proof}
Let $\psi \in C^{\infty}_c(\Om)$. Using lemma \ref{le05}, since $\phi > 0$ in $\Om$, we can find $ \alpha >0$ such that $v_\la \geq \alpha$ on $supp(\psi)$. Also, $t_{\epsilon} \rightarrow 1$ as $\epsilon \rightarrow 0+$, where $t_{\epsilon}$ is the unique positive real number corresponding to $(v_\la+\epsilon \psi)$ such that $t_\epsilon (v_\la+\epsilon \psi) \in \mc N^{-}_{\la}$. Then, by lemma \ref{le03} we have
\begin{equation*}
\begin{split}
0 & \leq \lim\limits_{\epsilon \rightarrow 0}\frac{I(t_\e(v_\la+\epsilon\psi)) - I(v_\la)}{\epsilon} \leq \lim \limits_{\epsilon \rightarrow 0}\frac{I(t_{\epsilon}(v_\la+\epsilon\psi)) - I(t_{\epsilon} v_\la)}{\epsilon}\\
& = \int_Q \frac{(v_\la(x)-v_\la(y))(\psi(x)-\psi(y))}{|x-y|^{n+2s}} dxdy - \la \int_\Om a(x) v_{\la}^{-q}\psi dx - \int_\Om  v_{\la}^{2^*_s-1}\psi dx.
\end{split}
\end{equation*}
Since $\psi \in C^{\infty}_c(\Om)$ is arbitrary, we conclude that $v_{\la}$ is positive weak solution of $(P_\la)$.\QED
\end{proof}
\noi {\bf Proof of Theorem \ref{thm2.4}:} Now the proof of theorem \ref{thm2.4} follows from proposition \ref{prp4.2}, lemma \ref{lem4.3} and proposition \ref{prp5.4}.

\section{Regularity of  weak solutions}

In this section, we shall prove some regularity properties of positive weak solutions of $(P_{\la})$. We begin with the following lemma.

\begin{Lemma}
Suppose $u$ is a  weak solution of $(P_{\la})$, then for each $w \in X_0$, it satisfies
\[ a(x)u^{-q}w \in L^{1}(\Om) \text{ and } \int_Q \frac{(u(x)-u(y))(w(x)-w(y))}{|x-y|^{n+2s}}~dxdy -  \int_\Om \left(\la a u^{-q} +  u^{2^{s}_{*}-1}\right)w dx = 0. \]
\end{Lemma}
\begin{proof}
Let $u$ be a  weak solution of $(P_{\la})$ and $w \in X_{+}$. By lemma $\ref{lem2.1}$, we obtain a sequence $\{w_k \} \in X_{0}$ such that $\{w_k\} \rightarrow w$ strongly in $X_0$, each $w_k$ has compact support in $
\Om$ and $0 \leq w_1 \leq w_2 \leq \ldots$ Since each $w_k$ has compact support in $\Om$ and $u$ is a positive weak solution of $(P_\la)$, for each $k$ we obtain
\[  \la \int_\Om a(x) u^{-q}w_k~dx = \int_Q \frac{(u(x)-u(y))(w_k(x)-w_k(y))}{|x-y|^{n+2s}} ~dxdy - \int_\Om u^{2^{s}_{*}-1}w_k~ dx .\]
Using monotone convergence theorem, we obtain $a(x)u^{-q}w \in L^{1}(\Om)$  and
\[\la \int_\Om a(x) u^{-q}w~dx = \int_Q \frac{(u(x)-u(y))(w(x)-w(y))}{|x-y|^{n+2s}}~ dxdy -\int_\Om u^{2^{s}_{*}-1}w~ dx = 0. \]
If $w \in X_0$, then $w = w^+ - w^-$ and $w^+, w^- \in X_{+}$. Since we proved the lemma for each $w \in X_+$, we obtain the conclusion.\QED
\end{proof}
\begin{Lemma}\label{lem6.2}
Let $u$ be a  weak solution of $(P_{\la})$. Then $u \in L^{\infty}(\Om)$.
\end{Lemma}
\begin{proof} We follow \cite{bss}.
We use the following inequality known for fractional Laplacian
\[(-\De)^s\varphi(u) \leq \varphi^{\prime}(u)(-\De)^s u, \]
where $\varphi$ is a convex and differentiable function. We define
$$\varphi(t) = \varphi_{T,\beta}(t)
\left\{
	\begin{array}{ll}
		0 & \mbox{if } t\leq 0 \\
		t^{\beta} & \mbox{if } 0<t<T\\
        \beta T^{\beta -1}(t-T) + T^{\beta} & \mbox{if } t\geq T
	\end{array}
\right.$$
where $\beta \geq 1$ and $T > 0$ is large. Then $\varphi$ is Lipschitz with constant $M = \beta T^{\beta -1}$ which gives $\varphi \in X_0$. Thus,
\begin{equation}\label{eqreg1}
\begin{split}
\|\varphi(u)\|  = \left( \int_{Q} \frac{|\varphi(u(x))-\varphi(u(y))|^2}{|x-y|^{n+2s}} \right)^{\frac{1}{2}}
 \leq \left( \int_{Q} \frac{M^2|u(x)-u(y)|^2}{|x-y|^{n+2s}} \right)^{\frac{1}{2}} = M^2\|u\|.
\end{split}
\end{equation}
Using $\|\varphi(u)\| = \|(-\De)^{s/2}\varphi(u)\|_{2}$, we obtain
\[ \int_{\Om}\varphi(u) (-\De)^s \varphi(u) = \|\varphi(u)\|^2 \geq S \|\varphi(u)\|^2_{2^*_s},
 \]
 where $S$ is as defined in section $1$. Since $\varphi$ is convex and $\varphi(u)\varphi^{\prime}(u) \in X_0$, we obtain
 \begin{equation}\label{eqreg2}
 \begin{split}
 \int_{\Om}\varphi(u) (-\De)^s \varphi(u) & \leq \int_{\Om}\varphi(u)\varphi^{\prime}(u) (-\De)^s u\\
 & = \int_{\Om}\varphi(u)\varphi^{\prime}(u) (\lambda a(x)u^{-q}+ u^{2^*_s-1})dx.
 \end{split}
\end{equation}
Therefore using \eqref{eqreg1} and \eqref{eqreg2}, we obtain
\begin{equation}\label{eqreg3}
\|\varphi(u)\|^2_{2^*_s} \leq C \int_{\Om}\varphi(u)\varphi^{\prime}(u)(\lambda a(x)u^{-q}+ u^{2^*_s-1})dx,
\end{equation}
for some constant $C$. We have $u \varphi^{\prime}(u)\leq \beta \varphi(u)$ and $\varphi^{\prime}(u)\leq \beta(1+\varphi(u))$ which gives
\begin{equation*}\label{eqreg5}
\begin{split}
&\int_{\Om}\varphi(u)\varphi^{\prime}(u)(\lambda a(x)u^{-q}+ u^{2^*_s-1})dx\\
& =  \int_{\Om} \left(\lambda a(x)\varphi(u)\varphi^{\prime}(u)u^{-q}dx + \varphi(u)\varphi^{\prime}(u)u^{2^*_s-1}\right) dx\\
& \le  \la \|a\|_{\infty}\beta \int_{\Om}\varphi(u)u^{-q}(1+\varphi(u)) + \beta \int_{\Om} (\varphi(u))^2 u^{2^*_s-2}.
\end{split}
\end{equation*}
Thus from \eqref{eqreg3}, we obtain
\begin{equation*}
\left(\int_{\Om}|\varphi(u)|^{2^*_s}\right)^{2/2^*_s} \leq C_2\beta \left( \la \|a\|_{\infty} \int_{\Om}(\varphi(u)u^{-q} + (\varphi(u))^2u^{-q}) + \int_{\Om}(\varphi(u))^2 u^{2^*_s-2} \right).\\
\end{equation*}
where $C_2= C\max\{\la \|a\|_{\infty},1\}$. Next we claim that $u \in L^{\beta_1 2^*_s}(\Om)$, where $\beta_1 = 2^*_s/2$. Fixing some $K$ whose appropriate value is to be determined later and taking $r = \beta_1, s = 2^*_s/(2^*_s-2)$, we obtain
\begin{equation*}
\begin{split}
\int_{\Om}(\varphi(u))^2u^{2^*_s-2} &= \int_{u\leq K}(\varphi(u))^2u^{2^*_s-2} + \int_{u > K}(\varphi(u))^2u^{2^*_s-2}\\
& \leq K^{2^*_s-2} \int_{u \leq K} (\varphi(u))^2 + \left( \int_{\Om}(\varphi(u))^{2^*_s} \right)^{2/2^*_s} \left( \int_{u>K} u^{2^*_s} \right)^{(2^*_s-2)/2^*_s}.
\end{split}
\end{equation*}
Using Monotone Convergence theorem, we choose $K$ such that
\[ \left(\int_{u>K}u^{2^*_s}\right)^{(2^*_s-2)/2^*_s} \leq \frac{1}{2C_2\beta} \]
and this gives
\begin{equation}\label{eqreg6}
\left(\int_{\Om} (\varphi(u))^{2^*_s} \right)^{2/2^*_s} \leq 2C_2\beta \left( \int_{\Om}(\varphi(u))u^{-q} + \int_{\Om}(\varphi(u))^2u^{-q} +K^{2^*_s-2} \int_{u\leq K}(\varphi(u))^2 \right).
\end{equation}
Using $\varphi_{T,\beta_1}(u)\leq u^{\beta_1}$ in left hand side of \eqref{eqreg6} and then letting $T \rightarrow \infty$ in right hand side, we obtain
\[ \left(\int_{\Om}u^{2^*_s \beta_1}\right)^{2/2^*_s} \leq 2C_2\beta_1 \left( \int_{\Om}u^{\frac{2^*_s}{2}-q} + \int_{\Om}u^{2^*_s-q} + K^{2^*_s-2}\int_{\Om}u^{2^*_s}\right) \]
since $2\beta_1 = 2^*_s$. This proves the claim. Again, from \eqref{eqreg5}, using $\varphi_{T,\beta}(u)\leq u^{\beta}$ in left hand side and then letting $T \rightarrow \infty$ in right hand side, we obtain
\begin{equation}\label{eqreg7}
\begin{split}
\left(\int_{\Om}u^{2^*_s \beta}\right)^{2/2^*_s} &\leq C_2\beta \left( \int_{\Om}u^{\beta -q} + \int_{\Om}u^{2\beta-q} + K^{2^*_s-2}\int_{\Om}u^{2 \beta+2^*_s-2}\right)\\
& \leq C_2^{\prime}\beta \left( 1 + \int_{\Om}u^{2\beta-q} + K^{2^*_s-2}\int_{\Om}u^{2 \beta+2^*_s-2}\right),
\end{split}
\end{equation}
where $C_2^{\prime}> 0$ is a constant. Now we see that
\begin{equation*}
\begin{split}
\int_{\Om}u^{2\beta-q} = \int_{u\geq 1}u^{2\beta-q}+ \int_{u<1}u^{2\beta-q} \leq  \int_{u \geq 1}u^{2\beta+2^*_s-2}+|\Om|.
\end{split}
\end{equation*}
Using this in \eqref{eqreg7}, with some simplifications, we obtain
\begin{equation}\label{eqreg8}
\left(1+\int_{\Om}u^{2^*_s \beta}\right)^{\frac{1}{2^*_s(\beta-1)}} \leq {C^{\frac{1}{2(\beta-1)}}_{\beta}}\left( 1+\int_{\Om}u^{2\beta+2^*_s-2} \right)^{\frac{1}{2(\beta-1)}},
\end{equation}
where $C_{\beta}= 4C_2^{\prime} \beta(1+|\Om|)$. For $m \geq 1$, let us define $\beta_{m+1}$ inductively by
\[ 2\beta_{m+1}+2^*_s-2 = 2^*_s \beta_m \]
that is $(\beta_{m+1} - 1)= \frac{2^*_s}{2}(\beta_m -1) = \left(\frac{2^*_s}{2}\right)^m(\beta_1-1)$. Hence, from \eqref{eqreg8} it follows that
\begin{equation*}
\left(1+\int_{\Om}u^{2^*_s \beta_{m+1}}\right)^{\frac{1}{2^*_s(\beta_{m+1}-1)}} \leq {C^{\frac{1}{2(\beta_{m+1}-1)}}_{\beta_{m+1}}}\left( 1+\int_{\Om}u^{2^*_s\beta_m} \right)^{\frac{1}{2^*_s(\beta_m-1)}},
\end{equation*}
where $C_{\beta_{m+1}}= 4C_2^{\prime}\beta_{m+1}(1+|\Om|)$. Setting $D_{m+1} := \left( 1+\int_{\Om}u^{2^*_s\beta_m} \right)^{\frac{1}{2^*_s(\beta_m-1)}}$, we obtain
\begin{equation*}
\begin{split}
D_{m+1} &\leq 
 \{4C_2^{\prime}(1+|\Om|)\}^{\sum_{i=2}^{m+1}\frac{1}{2(\beta_i-1)}} \prod^{m+1}_{i=2}\left( 1+\left( \frac{2^*_s}{2}\right)^{i-1}(\beta_1-1) \right)^{\frac{1}{2((2^*_s/2)^{i-1}(\beta_1-1))}}D_1.
\end{split}
\end{equation*}
Since 
$\lim\limits_{i \rightarrow \infty}\left(\left(\frac{2^*_s}{2}\right)^{i-1}(\beta_1-1)+1\right)^{\frac{1}{2((2^*_s/2)^{i-1}(\beta_1-1))}} = 1,$ there exists $C_3 >1$ (independent of $\beta_i$'s) such that
$$D_{m+1} \leq \{4C_2^{\prime}(1+|\Om|)\}^{\sum_{i=2}^{m+1}\frac{1}{2(\beta_i-1)}}C_3D_1.$$
However, $\sum_{i=2}^{m+1}\frac{1}{2(\beta_i-1)} = \frac{1}{2(\beta_1-1)}\sum_{i=2}^{m+1}\left( \frac{2^*_s}{2} \right)^i$ and $\left(\frac{2^*_s}{2}\right) <1$ implies that there exists a constant $C_4>0$ such that
$D_{m+1}\leq C_4D_1$, that is
\begin{equation}\label{eqreg10}
\left(1+ \int_{\Om}u^{2^*_s(\beta_{m+1})} \right)^{\frac{1}{2^*_s(\beta_{m+1}-1)}}\leq C_4D_1,
\end{equation}
where $2^*_s\beta_m \rightarrow \infty$ as $m \rightarrow \infty$. Let us assume $\|u\|_{\infty}> C_4D_1$. Then there exists $\eta > 0$ and ${\Om}^{\prime} \subset \Om$ such that $0<|{\Om}^{\prime}|< \infty$ and
\[u(x) > C_4D_1+\eta, \text{ for all } x \in {\Om}^{\prime}.\]
It follows that
\begin{equation*}
\liminf\limits_{\beta_m \rightarrow \infty}\left( \int_{\Om}|u|^{2^*_s\beta_m}+1\right)^{\frac{1}{2^*_s\beta_m -1}}\geq \liminf\limits_{\beta_m \rightarrow \infty}\left( C_4D_1+\eta \right)^{\frac{\beta_m}{\beta_m-1}}(|{\Om}^{\prime}|)^{\frac{1}{2^*_s(\beta_m-1)}} = C_4D_1+\eta
\end{equation*}
which contradicts \eqref{eqreg10}. Hence, $\|u\|_{\infty}\leq C_4D_1$ that is $u \in L^{\infty}(\Om)$. \QED
\end{proof}

\begin{Theorem}
Let $u$ be a positive solution of $(P_{\la})$. Then there exist $\alpha \in (0,s]$ such that $u \in C_{loc}^{\alpha}(\Om^{\prime})$, for all $\Om^{\prime} \Subset \Om$.
\end{Theorem}
\begin{proof}
Let $\Om^{\prime} \in \Om$. Then using lemma \ref{le05} and above regularity result, for any $\psi \in C_{c}^{\infty}(\Om)$ we obtain
\begin{equation*}
 \la \int_{\Om^{\prime}}u^{-q}\psi dx + \int_{\Om^{\prime}}u^{2^*_s-1}\psi  dx\leq \la \int_{\Om^{\prime}}\phi_{1}^{-q}\psi  dx+ \|u\|_{\infty}^{2^*_s-1} \int_{\Om^{\prime}}\psi dx \leq C \int_{\Om^{\prime}}\psi dx
\end{equation*}
for some constant $C>0$, since we can find $k>0$ such that $\phi_1>k$ on $\Om^{\prime}$. Thus we have $|(-\De)^su|\leq C$ weakly on $\Om^{\prime}$. So, using theorem 4.4 of \cite{Asm} and applying a covering argument on inequality in corollary 5.5 of \cite{Asm}, we can prove that there exist $\alpha \in (0,s] $ such that $u \in C_{loc}^{\alpha}(\Om^{\prime})$, for all $\Om^{\prime} \Subset \Om$.\QED
\end{proof}

\end{document}